\documentclass{amsart}

\usepackage{hyperref, url}
\usepackage{amsmath,amssymb,amsfonts,amsthm}
\usepackage{tikz}

\newtheorem{thm}{Theorem}[section]
\newtheorem{lem}[thm]{Lemma}
\newtheorem{prop}[thm]{Proposition}
\newtheorem{cor}[thm]{Corollary}
\theoremstyle{definition}
\newtheorem{rem}[thm]{Remark}

\newtheorem{alg}[thm]{Algorithm}
\newtheorem{ass}[thm]{Assumptions}

\newcommand{\core}{\mbox{\upshape{core}}(I)}
\newcommand{\h}{\mbox{\upshape{ht} }}

\begin{document}


\title{A Formula for the Core of Certain \\ Strongly Stable Ideals}
\author{Bonnie Smith}
\address{Department of Mathematics, University of Kentucky, Lexington, KY  40506, USA}
\email{bonnie.smith@uky.edu}

\begin{abstract}
The core of an ideal is the intersection of all of its reductions.  The core has geometric significance coming, for example, from its connection to adjoint and multiplier ideals.  In general, though, the core is is difficult to describe explicitly.  In this paper, we investigate a particular family of strongly stable ideals.  We prove that ideals in this family satisfy an Artin-Nagata property, yet fail to satisfy other, stronger standard depth conditions.  We then show that there is a surprisingly simple explicit formula for the core of these ideals.
\end{abstract}



\maketitle

\begin{section}{Introduction}
The notion of the \emph{core} of an ideal was introduced by Rees and Sally \cite{ReesSally}, who defined the core of an ideal $I$ to be the intersection of all reductions of $I$.  An ideal $J \subset I$ is a \emph{reduction} of $I$ if $JI^r=I^{r+1}$ for some $r \geq 0$, or, equivalently, if the extension of Rees algebras $R[Jt]=R \oplus Jt \oplus J^2t^2 \oplus \ldots \subset R[It]$ is module-finite.  Reductions  are also connected to integral closure in the following way:  in a Noetherian ring, an ideal $J \subset I$ is a reduction of $I$ if and only if $\overline{J} = \overline{I}$, where $\overline{I}$ denotes the integral closure of $I$.  The smallest $r$ for which $JI^r=I^{r+1}$ is the \emph{reduction number} of $I$ with respect to $J$, denoted $r_J(I)$.  Intuitively, a minimal reduction (with respect to inclusion) $J$ of $I$ can be thought of as a simplification of $I$ which carries much of the information about $I$, with the invariant $r_J(I)$ providing a measure of how closely $J$ and $I$ are related.  The core, in turn, encodes properties which are common to all minimal reductions of the ideal.  Reductions were first studied by Northcott and Rees \cite{Northcott}, in the setting of a Noetherian local ring with infinite residue field $k$.  They showed that any ideal $I$ which is not its own unique minimal reduction will have infinitely many minimal reductions, but that these reductions will all have the same minimal number of generators.  This number is the \emph{analytic spread} of $I$, denoted $\ell(I)$ or simply $\ell$, which is equal to the dimension of the special fiber ring $\mathcal{F}(I)=R[It] \otimes k$.  From this fact and \cite{L-S} follows a first observation about the core:  if $I$ is an ideal in a regular local ring with infinite residue field, then $\overline{I^\ell} \subset \core$.  The analytic spread is an important invariant which plays a key role in the study of the core. 

The core is closely related to the adjoint ideal defined by Lipman \cite{Lipman}.  In settings where both are defined, adjoint ideals coincide with multiplier ideals, fundamental tools in Algebraic Geometry which encode information about singularities.  (See for example \cite{Laz} and its references.)  Lipman's result (with a generalization by Ulrich \cite{BerndAdjoint}) shows that, for an ideal $I$ in a regular local ring, adj$(I^\ell) \subseteq \core$;  and equality has been shown under certain conditions by Huneke and Swanson \cite{HunekeSwanson}, and in \cite{HyrySmith}, \cite{Angie}, \cite{KPU}, \cite{ formula} and \cite{PUV}.  Hyry and Smith \cite{HyrySmith} showed that the core also has further geometric significance:  if one were able to determine the specific shape of certain cores, this would lead to a proof of an open conjecture of Kawamata about nonvanishing of global sections.  Additionally, Fouli, Polini and Ulrich \cite{FPU} have proven that the core of the ideal of a finite set of points in $\mathbb{P}^r$ reads information about the position of the points.  Specifically, one can tell from the core whether or not the points have the Cayley-Bacharach property (that is, whether the Hilbert function of $n-1$ of the $n$ points is independent of which point is excluded).

Unfortunately, the core is quite difficult to describe explicitly, as it is in principle an intersection of infinitely many ideals.  However, Corso, Polini and Ulrich \cite{structure} have shown, under mild assumptions, that the core is actually a finite intersection of reductions, and that one may obtain the core of $I$ by intersecting finitely many general reductions of $I$.  The assumptions on the ideal $I$ needed for this result are (i) that $I$ have the $G_\ell$ property, where $\ell$ is the analytic spread of $I$, and (ii) that $I$ be weakly $(\ell-1)$-residually $S_2$.  (See Section 2 for definitions.)  These assumptions cannot be weakened:  \cite[Example 4.11]{structure} shows that the core of $I$ may be strictly smaller than the intersection of all general reductions of $I$, if $I$ does not have the $G_\ell$ property.  In contrast, suppose instead that $I$ is a monomial ideal which does satisfy assumptions (i) and (ii) above, and suppose also that $\ell(I)$ is equal to the dimension of $R$ (the largest value of $\ell(I)$ possible).  In this case, Polini and Ulrich \cite{Monom} have shown that there is a characterization of the core of $I$ in terms of a single reduction of $I$.  Such a characterization was first given by Polini, Ulrich and Vitulli \cite{PUV} in the case of zero-dimensional ideals.  We shall make use of this result of Polini and Ulrich, the statement of which is given in Section 2, in the proof of our main result (Theorem \ref{core thm}).

Corso, Polini and Ulrich \cite[Remark 5.1]{structure} have shown that the core of any monomial ideal is also monomial.  Given this fact, one might hope for a combinatorial description of the core of a monomial ideal, such as that for the adjoint of a monomial ideal given by Howald \cite{Howald}.  However, one faces the immediate obstacle that the core is usually not integrally closed.  Therefore, no description of the core of a monomial ideal in terms of the Newton polyhedron (such as one has for the adjoint of a monomial ideal) is possible.  The problem of describing the core becomes tractable, though, if we restrict our attention to particular classes of monomial ideals.  The main object of this paper is to give an explicit formula for the core of such a class of ideals, namely the strongly stable ideals of degree two having the $G_\ell$ property.  Let $R=k[X_1,\ldots,X_d]$ be a polynomial ring over a field $k$. A monomial ideal $I \subset R$ is \emph{strongly stable} if $m X_i / X_j \in I$ for all $i < j$, for every monomial $m \in I$ and every $j$ such that $X_j \mid m$.  A strongly stable ideal $I \subset R$ of degree two (that is, generated by elements of degree two) is necessarily non-squarefree.  For example, $X_1^2$ must be in $I$.  However, such an ideal can be thought of as the edge ideal of a graph with loops---an interpretation which we shall exploit in determining the analytic spread of $I$.  Strongly stable ideals of degree two have been studied recently by Corso and Nagel \cite{Ferrers1, Ferrers2} in their work on Ferrers ideals and their specializations.  Strongly stable ideals in general have received much attention because of their connection (in characteristic zero) with generic initial ideals (see for example \cite{Eis}).

We now indicate the structure of the paper.  In Section 2 we give necessary definitions and background.  In Section 3 we study properties of strongly stable ideals in our class, with an eye towards applying Polini and Ulrich's characterization of the core, which was mentioned above.  We show that our ideals satisfy the Artin-Nagata property $AN_{d-1}$, a condition on the Cohen-Macaulayness of certain residual intersections (see Section 2 for definitions)---a fact which is interesting in and of itself, since our ideals fail to have stronger standard depth conditions.  In Section 4 we lay the technical groundwork needed to prove our main result (Theorem \ref{core thm}), a surprisingly simple explicit formula for the core of our ideals, which is given in Section 5.  Throughout the paper, except where specifically stated, we shall work in a polynomial ring over a field of characteristic zero.  With this assumption on the characteristic of $k$, the core of a strongly stable ideal is also strongly stable.  We use this fact in the proof of one inclusion of our main theorem.

\end{section}

\begin{section}{Preliminaries}

We begin by reviewing some definitions and results relating to residual properties of ideals.  Let $R$ be a Cohen-Macaulay ring, and let $s$ be an integer.  An ideal $I \subset R$ is said to have the \emph{$G_s$ property} if $\mu(I_\mathfrak{p}) \leq \h \mathfrak{p}$ for every $p \in$ Spec$(R)$ such that $\mathfrak{p} \supset I$ and ht $\mathfrak{p} < s$.  An $i$-\emph{residual intersection} of $I$ is an ideal $J=\mathfrak{a}:I$ such that $\mathfrak{a} \subsetneq I$ is $i$-generated and $\h J \geq i \geq \h I$.  Residual intersections can be thought of as generalizations of linked ideals, where for example the heights of $J=\mathfrak{a}:I$ and $I$ need not be the same.  An example of how residual intersections arise was shown by Ulrich \cite[Proposition 1.11]{Hadley} (see also the formulation in \cite[Remark 2.7]{MarkJ}).  He proved, under mild assumptions, that if $\mathfrak{a}$ is any minimal reduction of an ideal $I$, then $\mathfrak{a}:I$ is an $\ell(I)$-residual intersection of $I$, where $\ell(I)$ is the analytic spread of $I$.  An ideal $I$ is said to be $s$-\emph{residually} $S_2$ if, for every $i \leq s$, for every $i$-residual intersection $J$ of $I$, the ring $R/J$ satisfies the $S_2$ property.  A ring $R$ satisfies the $S_2$ property if depth $R_{\mathfrak{p}} \geq \min\{ 2, \h \mathfrak{p}\}$ for every prime $\mathfrak{p}$ of $R$.  A related condition is the \emph{Artin-Nagata property} $AN_s$, which was defined by Ulrich \cite{Hadley}.  An ideal $I$ satisfies the Artin-Nagata property $AN_s$ if, for every $i \leq s$, for every $i$-residual intersection $J$ of $I$, the ring $R/J$ is Cohen-Macaulay.  Clearly this is stronger than the $s$-residually $S_2$ condition, as Cohen-Macaulay rings satisfy the $S_2$ property. 

Another notion which is closely connected to those above is the sliding depth property, which was defined by Herzog, Vasconcelos and Villarreal \cite{HVV}.  An ideal $I$ which is minimally generated by elements $f_1,\ldots,f_n$ is said to have \emph{sliding depth} if depth $H_i(f_1,\ldots,f_n) \geq d-n+i$ for every $i$, where $H_i(f_1,\ldots,f_n)$ denotes the $i$th Koszul homology module of $f_1,\ldots,f_n$.  A result of Herzog, Vasconcelos and Villarreal \cite[Theorem 3.3]{HVV}, with a modification by Ulrich \cite{BNotes}, shows that if an ideal $I$ satisfies the $G_s$ property and has sliding depth, then $I$ satisfies the Artin-Nagata property $AN_{s-1}$.

\medskip

Now let $R=k[X_1,\ldots,X_d]$ be a polynomial ring over an infinite field $k$ (of any characteristic).  Polini and Ulrich \cite{Monom} have shown that the core of certain monomial ideals can be characterized in terms of a single reduction.  We state their result here.  The term \emph{general locally minimal reduction} of $I$ denotes any ideal of the form $K=(f_1,\ldots,f_\ell,h^d)$, where $f_1,\ldots,f_\ell$ are general linear combinations of the generators of $I$, and $h$ is any element of $I$.  (See also \cite[3.3]{PUV} for a discussion of general locally minimal reductions in the zero-dimensional case.)  The result of Polini and Ulrich is as follows:  Let $I \subset R$ be a monomial ideal having the $G_\ell$ property.  Suppose also that $I$ is weakly $(\ell-1)$-residually $S_2$ and that $\ell=d$.  Then if $K$ is any general locally minimal reduction of $I$, then the largest monomial ideal contained in $K$ (denoted mono$(K)$) is equal to the core of $I$.

From a computational standpoint, the ideal mono$(K)$ can easily be computed with the computer algebra system Macaulay 2, using the commands \\ \tt minimalReduction \rm and \tt monomialSubideal\rm.  If $I$ is generated by monomials of the same degree, as in our case, then one may take $K=(f_1,\ldots,f_\ell)$, as this is (globally) a minimal reduction of $I$. 

\begin{ass}\label{ring}  From now on we let $R=k[X_1,\ldots,X_d]$, where $k$ is a field of characteristic zero.  Except where otherwise stated, $I \subset R$ will be a strongly stable ideal of degree two, and we shall assume that $X_1X_d \in I$.
\end{ass}

\begin{rem}\label{d}  For the purposes of computing the core, the assumption that $X_1X_d$ is in $I$ imposes no restriction.  Since $I$ is strongly stable, if $X_1X_d \notin I$, then $I$ can be thought of as the extension of an ideal $I' \subset R'=k[X_1,\ldots,X_{d'}]$, where $d' <d$ and $X_1X_{d'} \in I'$.  In this case, core$(I)=$ core$(I')R$, so that it suffices to compute core$(I')$.
\end{rem}

As explained in \cite{Ferrers2, Ferrers1}, to each ideal $I$ we may associate a tableau $T_I$.  Following the example of Corso and Nagel, we use the convention that $T_I$ has a square in the $i$th row, $j$th column whenever $X_iX_j \in I$ and $i \leq j$.  With this convention, $g= \h I$ is the number of rows in $T_I$, while, with Assumption \ref{ring}, $d$ is the number of columns in $T_I$.  Figure 1 depicts the tableau of the height 4 ideal 
\begin{align}
I=(X_1^2,X_1X_2,&X_1X_3,X_1X_4,X_1X_5,X_1X_6,X_2^2,X_2X_3, \notag \\
&X_2X_4,X_2X_5,X_2X_6,X_3^2,X_3X_4,X_3X_5,X_3X_6,X_4^2) \notag 
\end{align}
 in the ring $ R=k[X_1,X_2,X_3,X_4,X_5,X_6]$.

\begin{figure}\label{T}

\begin{center}

\begin{tikzpicture}[scale=.6]

\fill[gray!30!white] (4.2,1.8) rectangle (4.8,1.2);

\draw (-1,4)--(5,4);
\draw (-1,3)--(5,3);
\draw (0,2)--(5,2);
\draw (1,1)--(5,1);
\draw (2,0)--(3,0);

\draw (-1,4)--(-1,3);
\draw (0,4)--(0,2);
\draw (1,4)--(1,1);
\draw (2,4)--(2,0);
\draw (3,4)--(3,0);
\draw (4,4)--(4,1);
\draw (5,4)--(5,1);

\draw (-1,4) node[anchor=south west]{$X_1$};
\draw (0,4) node[anchor=south west]{$X_2$};
\draw (1,4) node[anchor=south west]{$X_3$};
\draw (2,4) node[anchor=south west]{$X_4$};
\draw (3,4) node[anchor=south west]{$X_5$};
\draw (4,4) node[anchor=south west]{$X_6$};

\draw (-1,3) node[anchor=south east]{$X_1$};
\draw (-1,2) node[anchor=south east]{$X_2$};
\draw (-1,1) node[anchor=south east]{$X_3$};
\draw (-1,0) node[anchor=south east]{$X_4$};

\end{tikzpicture}

\end{center}

\caption{The tableau associated to a strongly stable ideal $I$ of degree two}

\end{figure}

Before proceeding, we observe that, with $I$ and $R$ as in Assumptions \ref{ring}, the analytic spread $\ell(I)$ is equal to the dimension $d$ of $R$ by a result of Villarreal:  as mentioned in Section 1, $I$ can be thought of as the edge ideal of a non-simple graph $G$.  The \emph{edge ideal} of a graph $G$ on vertices $v_1,\ldots,v_d$ is the ideal $(\{X_iX_j \mid (v_i,v_j) \mbox{ is an edge of } G\})$.  Another algebraic object which one can associate to $G$ is the \emph{monomial subring} of $G$, $k[G]:=k[\{X_iX_j \mid (v_i,v_j)$ is an edge of $ G\}]$.  If $I$ is the edge ideal of $G$, then $k[G] \cong \mathcal{F}(I)$, where $\mathcal{F}(I)$ is the special fiber ring of $I$.  Villarreal \cite[Corollary 8.2.13 and Exercise 8.2.16]{Vill} showed that if $G$ is any connected, non-bipartite graph (possibly having loops) on $d$ vertices, then $\dim k[G] = d$.  As $\ell(I)=\dim \mathcal{F}(I)$, this shows that $\ell(I)=d$ whenever $I$ is the edge ideal of such a graph.  Finally, note that strongly stable ideals satisfying Assumptions \ref{ring} are always non-bipartite and connected.

\medskip

We conclude this section with a result about strongly stable ideals, not necessarily of degree two, which we shall use in the proof of our main result, Theorem \ref{core thm}.  Since we assume that $k$ has characteristic zero, the strongly stable ideals of $R=k[X_1,\ldots,X_d]$ are precisely those ideals which are fixed under the action of invertible upper-triangular $d \times d$ matrices with entries in $k$.  (See for example \cite[Proposition 2.3]{Miller}.)  Analogously, monomial ideals are precisely the ideals which are fixed under the action of invertible diagonal matrices.  This fact has been used to show that the core of a monomial ideal is also monomial (\cite[Remark 5.1]{structure}).  Here we use the same approach to show that the core of a strongly stable ideal is also strongly stable. 

\begin{prop}\label{core ss} Let $I \subset R$ be strongly stable.  Then $\core$ is also strongly stable.
\end{prop}

\begin{proof} Let $A=(a_{ij})$ be an invertible $d \times d$ upper-triangular matrix with entries in $k$.  Let $\varphi_A:R \rightarrow R$ be the homomorphism induced by $A$, that is, the homomorphism defined by $\varphi_A(X_j)=\sum_{i=1}^da_{ij}X_i$.  Since $I$ is strongly stable, $\varphi_A(I)=I$.  Therefore, by the persistence of integral closure (see \cite{HSBook}), $\overline{I}= \varphi_A \circ \varphi_{A^{-1}} (\overline{I}) \subseteq \varphi_A (\overline{\varphi_{A^{-1}}(I)})  = \varphi_A(\overline{I}) \subseteq \overline{\varphi_A(I)}=\overline{I}$.  Thus $\varphi_A(\overline{I})=\overline{I}$, which shows that $\overline{I}$ is strongly stable.  Now let $J$ be a reduction of $I$.  Observe that $\varphi_A(J) \subseteq \varphi_A(I)=I$, and, by what we have just shown, $\overline{I}=\varphi_A(\overline{I})=\varphi_A(\overline{J}) \subseteq \overline{\varphi_A(J)}$.  Hence $\varphi_A(J)$ is a reduction of $I$.  Furthermore, $J'=\varphi_{A^{-1}}(J)$ is also a reduction of $I$, and $\varphi_A(J')=J$.  Therefore $\core$ is fixed under the action of $A$, which shows that $\core$ is strongly stable.
\end{proof}

\end{section}

\begin{section}{Artin-Nagata Properties of a Strongly Stable Ideal}

In this section we investigate properties of a strongly stable ideal $I$ as in Assumption \ref{ring}.  Recall from Section 2 that, if $I$ satisfies the $G_d$ property and is weakly $(d-1)$-residually $S_2$, then the core of $I$ may be characterized in terms of a single reduction of $I$.  We begin by giving a necessary and sufficient condition under which $I$ will satisfy the $G_d$ property.  Somewhat surprisingly, this is a condition on only a single element of $R$.

\begin{prop}\label{Gd}

Let $I \subset R$ be a strongly stable ideal of degree two of height $g$.  Then $I$ satisfies the $G_d$ property if and only if $X_{g-1}X_d \in I$.

\end{prop}

\begin{proof}
Suppose first that $X_{g-1}X_d \notin I$.  Set $s= \max \{ j: X_{g-1}X_j \in I\}$, and set $\mathfrak{p}=(X_1,\ldots,X_s)$.  Since $\h I =g$ and $I$ is strongly stable, $X_g^2 \in I$, hence $X_{g-1}X_g \in I$.  Therefore $g \leq s < d$, and thus $\mathfrak{p} \supset I$ and $\h \mathfrak{p} < d$.  We will show that $\mu(I_\mathfrak{p})>s = \h \mathfrak{p}$.  Set $t=$ min$\{i: X_iX_{s+1} \notin I\} \leq g-1$.  For $i \geq t$, $X_iX_j \notin I$ for all $j > s$ since $I$ is strongly stable, hence $X_{i} \notin I_\mathfrak{p}$.  For $i < t$, $X_i \in I_{\mathfrak{p}}$.  Therefore a minimal generating set of $I_\mathfrak{p}$ contains the set $\{X_1,\ldots,X_{t-1},X_{t}^2,\ldots,X_{g-1}^2,X_{g-1}X_g,\ldots,X_{g-1}X_s,X_g^2\}$, which has size $s+1$. 

Now suppose that $X_{g-1}X_d \in I$.  Let $\mathfrak{p} \supset I$ with ht $\mathfrak{p} < d$.  Let $X_j \notin \mathfrak{p}$.  Note that $X_iX_j \in I$ for all $i \leq g-1$, since $I$ is strongly stable.  Therefore, $X_i \in I_\mathfrak{p}$ for all $i \leq g-1$.  Write $n= \max\{b: X_gX_b \in I\}$.  If $X_k \notin \mathfrak{p}$ for some $k \leq n$, then $I_{\mathfrak{p}}=(X_1,\ldots,X_g)$.  Otherwise, if $\mathfrak{p} \supset (X_1,\ldots,X_n)$, then $\mu(I_\mathfrak{p})=n \leq \h \mathfrak{p}$.
\end{proof}

In terms of the tableau $T_I$ associated to $I$, Proposition \ref{Gd} shows that we need only look for the square in the last column of the second row from the bottom in $T_I$.  The presence of the shaded square in Figure 1 shows that the corresponding ideal has the $G_d$ property.  

The remainder of this section we devote to showing that every strongly stable ideal $I$ of degree two which satisfies the $G_d$ property satisfies the Artin-Nagata property $AN_{d-1}$.  (Recall that this is a stronger condition than the $(d-1)$-residually $S_2$ condition.)  We begin with a result about the saturation of $I$.

\begin{prop}\label{SD} If $I$ is a strongly stable ideal of degree two which satisfies the $G_d$ property, then the saturation of $I$ has the sliding depth property.

\end{prop}

\begin{proof}

Let $\mathfrak{m}=(X_1,\ldots,X_d)$, and let $g = \h I$.  By Proposition \ref{Gd}, we can write $I=(X_1,\ldots,X_{g-1})\mathfrak{m}+X_g(X_g,\ldots,X_\nu)$ for some $g \leq \nu \leq d$.  Let $I'$ be the saturation of $I$, which is given by
\[ I'= \begin{cases}
(X_1,\ldots,X_{g-1})+X_g(X_g,\ldots,X_\nu), & \mbox{ if } \nu < d \\
(X_1,\ldots,X_g), & \mbox{ if } \nu=d. 
\end{cases} \]
 If $\nu=d$, then $I'$ is generated by a regular sequence.  In this case it is clear that $I'$ has sliding depth since its Koszul complex is exact.  Now suppose instead that $\nu < d$.  Herzog, Vasconcelos and Villarreal \cite[Lemma 3.5]{HVV} have proven that an ideal $L \subset R$ has sliding depth if and only if, after going modulo a regular sequence in $L$, the corresponding statement holds.  Therefore, it suffices to show that $X_g(X_g,\ldots,X_\nu) \subset R'=k[X_g,\ldots,X_d]$ has sliding depth.  

We compare the Koszul complex of $\underline{\alpha}=X_g^2, X_gX_{g+1}, \ldots, X_gX_\nu$ with that of $\underline{\beta}=X_g,\ldots,X_\nu$.  As Herzog, Vasconcelos and Villarreal \cite[Section 1]{HVV} have pointed out, sliding depth can also be characterized in terms of the kernels of the Koszul maps.  Specifically, an ideal $L \subset R$ minimally generated by $f_1,\ldots,f_n$ has sliding depth if 
\begin{align} \mbox{depth Ker}(\partial_i) \geq \min\{d,d-n+i+1\} \mbox{ for all } i,
\end{align}
where $\partial_i$ is the $i$th Koszul map of $f_1,\ldots,f_n$.  

As the Koszul complex $K_\bullet(\underline{\beta})$ is exact, the kernels of the maps of $K_\bullet (\underline{\beta})$ satisfy equation (1).  Now observe that each map of $K_\bullet(\underline{\alpha})$ is simply $X_g$ times the corresponding map of $K_\bullet(\underline{\beta})$.  Since $X_g$ is a regular element, this shows that the kernel of the $i$th map of $K_\bullet(\underline{\alpha})$ is equal to the kernel of the $i$th map of $K_\bullet(\underline{\beta})$ with a degree shift.  Therefore $(\underline{\alpha})$ also satisfies equation (1), and hence has sliding depth.
\end{proof} 

\begin{thm}\label{ANI} Let $I$ be a strongly stable ideal of degree two.  If $I$ satisfies the $G_d$ property, then $I$ satisfies the Artin-Nagata property $AN_{d-1}$.

\end{thm}

\begin{proof}

Ulrich \cite[Remark 1.12]{Hadley} has shown the following:  let $s$ be an integer, let $\mathfrak{a} \subset I$ be ideals in a local Cohen-Macaulay ring with ht $\mathfrak{a}:I \geq s+1$, and suppose that $I$ satisfies the $G_s$ property.  Then if $I$ satisfies $AN_s$, then $\mathfrak{a}$ satisfies $AN_s$ as well.  

Recall that $I$ and its saturation $I'$ are equal locally on the punctured spectrum.  Thus if $I$ has the $G_d$ property, then clearly $I'$ does, as $I_{\mathfrak{p}}=(I')_{\mathfrak{p}}$ for all $\mathfrak{p}$ with $\h \mathfrak{p}< d$.  This fact also shows immediately that ht $I:I'=d$.  Therefore, it suffices to check that $I'$ satisfies the Artin-Nagata property $AN_{d-1}$.  A result of Herzog, Vasconcelos and Villarreal \cite[Theorem 3.3]{HVV} modified by Ulrich \cite{BNotes} shows that $I'$ satisfies the $AN_{d-1}$ property if it has sliding depth and satisfies the $G_d$ property.  Therefore Proposition \ref{SD} completes the proof.
\end{proof}

We conclude this section by remarking that many known examples of ideals which which satisfy the $AN_{d-1}$ property also satisfy a stronger depth condition on the powers of the ideal, namely that depth $R/I^j \geq \dim R/I -j+1$ for all $j$ with $1 \leq j \leq d- \h I$.  Many standard examples are even strongly Cohen-Macaulay (that is, all of their Koszul homology modules are Cohen-Macaulay), a property which was introduced by Huneke \cite{CraigSCM}.  In contrast, our strongly stable ideals have depth zero, but positive dimension (with the exception of the ideal $I=(X_1,\ldots,X_d)^2$), and therefore fail to satisfy these stronger depth conditions.

\end{section}

\begin{section}{The Diagonal Reduction}

In this section we introduce an ideal $J$ associated to a strongly stable ideal $I$ of degree two, which will be fundamental to our computation of the core of $I$.  This ideal $J$ has also been studied by Corso, Nagel, Petrovic and Yuen \cite{CNPY}.  The content of this section will lay the groundwork for one inclusion in our main result, Theorem \ref{core thm}.

Let $I$ be a strongly stable ideal of degree two.  We define the \emph{diagonal ideal associated to $I$} to be the ideal
\[J= \Big( \Big\lbrace \sum_{j=1}^{\beta_n} X_{j}X_{j+n-1}: n=1,\ldots, d \Big\rbrace \Big), \mbox{ where } \beta_n= \max \big\lbrace b: X_bX_{b+n-1} \in I\big\rbrace.\]  
Note that the $i$th generator of $J$ is the sum of the generators of $I$ which correspond to squares lying along the $i$th diagonal in the tableau $T_I$.  

Figure 2 shows the diagonals of the tableau associated to the ideal $I$ from Figure 1.  The diagonal ideal $J$ associated to this ideal $I$ is 
\begin{align}
J=(X_1^2+X_2^2+X_3^2+X_4^2,\ X_1X_2+X_2X_3+&X_3X_4, \notag \\
X_1X_3+X_2X_4+X_3X_5, \ X_1X_4&+X_2X_5+X_3X_6, \notag \\
&X_1X_5+X_2X_6,\ X_1X_6).  \notag
\end{align}

\begin{figure}\label{J}

\begin{center}

\begin{tikzpicture}[scale=.6]

\draw (-1,4)--(5,4);
\draw (-1,3)--(5,3);
\draw (0,2)--(5,2);
\draw (1,1)--(5,1);
\draw (2,0)--(3,0);

\draw[very thick] (-.6,3.6)--(2.6,.4) (.4,3.6)--(2.6,1.4) (1.4,3.6)--(3.6,1.4) (2.4,3.6)--(4.6,1.4) (3.4,3.6)--(4.6,2.4) (4.3,3.7)--(4.7,3.3);

\draw (-1,4)--(-1,3);
\draw (0,4)--(0,2);
\draw (1,4)--(1,1);
\draw (2,4)--(2,0);
\draw (3,4)--(3,0);
\draw (4,4)--(4,1);
\draw (5,4)--(5,1);

\draw (-1,4) node[anchor=south west]{$X_1$};
\draw (0,4) node[anchor=south west]{$X_2$};
\draw (1,4) node[anchor=south west]{$X_3$};
\draw (2,4) node[anchor=south west]{$X_4$};
\draw (3,4) node[anchor=south west]{$X_5$};
\draw (4,4) node[anchor=south west]{$X_6$};

\draw (-1,3) node[anchor=south east]{$X_1$};
\draw (-1,2) node[anchor=south east]{$X_2$};
\draw (-1,1) node[anchor=south east]{$X_3$};
\draw (-1,0) node[anchor=south east]{$X_4$};

\end{tikzpicture}

\end{center}

\caption{The diagonals of the tableau associated to $I$}

\end{figure}
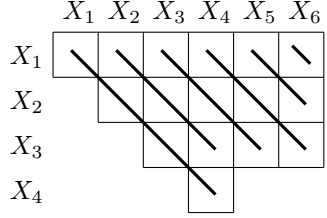

Our main result in this section is as follows:

\begin{prop}\label{Im J} Let $I$ be a strongly stable ideal of degree two having the $G_d$ property.  Let $g=\h I$, and let $\mathfrak{m}=(X_1,\ldots,X_d)$.  Let $J$ be the diagonal ideal associated to $I$.  Then $I\mathfrak{m}^{g-1} \subset J$.

\end{prop}

To prove Proposition \ref{Im J}, we will first impose a particular ordering on the elements of $I\mathfrak{m}^{g-1}$.  We then prove that $I \mathfrak{m}^{g-1} \subset J$ by strong induction proceeding according to the ordering.  The following algorithm produces the desired ordering:

\begin{alg}\label{J-alg} Let $S=\{M_i\}$ be the ordered set produced by the following algorithm:

\begin{enumerate}

\item Set $M_1=X_1X_d$, and set $i=1$.

\item Write $M_i=X_{h_i}X_{r_i}X_{t_i} \sigma_i$, where $h_i \leq r_i \leq t_i$ and $X_j | \sigma_i \ \Rightarrow \ j \geq t_i$. 

\begin{itemize}

\vspace{.05in}

\item If $h_i < r_i$ and $h_i < g$, set $M_{i+1}=X_{h_i+1}X_{r_i}^2X_{t_i} \sigma_i$.

\vspace{.05in}

\item Else if $h_i=g <r$, set $M_{i+1}=X_{h_i}X_{r_i-1}X_{t_i}\sigma_i$.

\vspace{.05in}

\item Else if $h_i=r_i$, set $M_{i+1}=X_{h_i-1}X_{t_i-1}\sigma_i$.

\end{itemize}

\vspace{.05in}

\item If $X_1^2 \nmid M_{i+1}$, set $i=i+1$ and return to (2).  Else if $X_1^2 \mid M_{i+1}$, done.

\end{enumerate} 

\end{alg}

We must prove that Algorithm \ref{J-alg} finishes in a finite number of steps.  Furthermore, for a given element $M_i$ in $S$, we must characterize which monomials of $R$ are elements of $S$ which precede $M_i$ in the ordering of the algorithm.  The next two results will accomplish both tasks.  First we shall need some additional notation.

For each $h=1,\ldots, g$, set $S_h:=\{ M_i \in S \mid h=h_i\}$.  That is, $S_h=\{M_i \in S \mid h=\min\{j \mid X_j \mid M_i\} \}$.  It follows easily by induction on $i$ that, for each $M_i \in S_h$, $\deg(M_i)=h+1$.  Note that $M_1=X_1X_d \in S_1$ and $\deg(M_1)=2$.  Now let $i > 1$ and assume the claim holds for $M_{i-1}$.  Write $M_{i-1} \in S_h$.  If $M_{i} \in S_{h+1}$, then $\deg(M_{i})=\deg(M_{i-1}) +1$;  if $M_i \in S_h$, then $\deg(M_i)=\deg(M_{i-1})$; and if $M_{i} \in S_{h-1}$, then $\deg(M_{i})=\deg(M_{i-1})-1$. 

\begin{rem}\label{alg} A trivial observation about Algorithm \ref{J-alg} which we shall invoke is the following:  if $M_i \in S_h$ and $M_{i+p} \in S_{h+k}$ for some $p >0$ and $k \in \mathbb{Z}$, then the set $\{M_i, M_{i+1}, \ldots, M_{i+p}\}$ contains at least one element from each of the sets $S_h, \ldots, S_{h+k}$.
\end{rem}

Write $S_h=\{M_{h,j} \mid j=1,\ldots,\nu_h\}$, where for each $j < \nu_h$, $M_{h,j}$ arises before $M_{h,j+1}$ in the algorithm.  We shall show that, for each $h$, this ordering on $S_h$ is actually the reverse lexicographic ordering.  For each $h=1,\ldots, g$, set $T_h=\big\lbrace N \in R \mid N \mbox{ is a monomial, } h=\min\{\alpha \mid X_\alpha \mid N\}, \deg(N)=h+1\big\rbrace$. 
For each $h$ write $T_h=\{N_{h,m}\}$, where $N_{h,m} < N_{h,m+1}$ in the reverse lexicographic order.  

\begin{lem}\label{Sh-Th} With notation as above, for each $h=1,\ldots,g$, $S_h$ and $T_h$ are equal as ordered sets.   

\end{lem}

\begin{proof}

Clearly $N_{h,1}=X_hX_d^h$ for each $h=1,\ldots, g$.  One can check that also $M_{h,1}=X_hX_d^h$ for each $h$.  (This is easily verified, as $M_{1,1},\ldots,M_{g,1}$ are the first $g$ elements produced by the algorithm.)  Also, clearly $S_h \subseteq T_h$ for each $h$;  hence for each $(h,j)$, there is some $m$ such that $M_{h,j}=N_{h,m}$.  Fix $h$, and fix $j < \nu_h$.  We will prove that $N_{h,m} \neq X_h^{h+1}$---that is, that $N_{h,m}$ is not the greatest element of $T_h$ in the reverse lexicographic ordering---and that $M_{h,j+1}=N_{h,m+1}$.  This will complete the proof.  Write $M_{h,j}=N_{h,m}=X_h^{n_h}X_{w_1}^{n_{w_1}} \ldots X_{w_s}^{n_{w_s}} \mbox{, where } h < w_1 < \ldots < w_s$.  
If $N_{h,m} \neq X_h^{h+1}$, then $N_{h,m+1}= X_hX_{w_1-1}^{n_{h}}X_{w_1}^{n_{w_1}-1}X_{w_2}^{n_{w_2}} \ldots X_{w_s}^{n_{w_s}}$.  The proof proceeds in three cases.

\emph{Case 1:} $n_h=1$.  We induct on $g-h$, and we shall need to treat two base cases separately.  If $h=g$, then the instructions of the algorithm stipulate that $M_{i+1}=X_hX_{w_1-1}X_{w_1}^{n_{w_1}-1}X_{w_2}^{n_{w_2}} \ldots X_{w_s}^{n_{w_s}}$.  Clearly this is the element of $S_h$ immediately following $M_{h,j}$;  that is, $M_{i+1}=M_{h,j+1}$.  It is also the element of $T_h$ which immediately follows $T_{h,m}$ in the reverse lexicographic ordering;  thus $M_{h,j+1}=T_{h,m+1}$ as claimed.  If $g-h=1$, then $M_{i+1}=X_gX_{w_1}^{n_{w_1}+1}X_{w_2}^{n_{w_2}} \ldots X_{w_s}^{n_{w_s}}$.  If $w_1=g$, then $M_{i+2}= X_{g-1}^2X_{w_1}^{n_{w_1-1}}X_{w_2}^{n_{w_2}} \ldots X_{w_s}^{n_{w_s}}$.  This is $M_{h,j+1}$, which is $N_{h,m+1}$ as claimed.
Otherwise, if $w_1 > g$, then $M_{i+2} =X_gX_{w_1-1}X_{w_1}^{n_{w_1}} \ldots X_{w_s}^{n_{w_s}}$, and $M_{i+3} =X_{g-1}X_{w_1-1}X_{w_1}^{n_{w_1}-1}X_{w_2}^{n_{w_2}}\ldots X_{w_s}^{n_{w_s}}$.  In this case $M_{i+3}= M_{h,j+1}$, which is $N_{h,m+1}$ as claimed.  

Now let $g-h >1$, and assume the claim holds for $h+1$.  That is, for all $j'< \nu_{h+1}$ such that $X_{h+1}^2 \nmid M_{h+1,j'}$, assume that, if $M_{h+1,j'}=N_{h+1,m'}$, then $M_{h+1,j'+1}=N_{h+1,m'+1}$.  If $w_1=h+1$, then $M_{i+1}=X_{h+1}^{h+2}X_{w_2}^{n_{w_2}}\ldots X_s^{n_s}$.  In this case, $M_{i+2}=X_h^2X_{w_1}^{n_{w_1}-1}X_{w_2}^{n_{w_2}} \ldots X_s^{n_s}=N_{h,m+1}$, and this is $M_{h,j+1}$ as claimed.  Otherwise, if $w_1 > h+1$, then $M_{i+1}=X_{h+1}X_{w_1}^{n_{w_1}+1}X_{w_2}^{n_{w_2}} \ldots X_s^{n_s}$.  In this case, $M_{i+2}$ will be in $S_{h+2}$.  Therefore, by Remark \ref{alg}, there must be at least one element of $S_{h+1}$ which falls between $M_{i+1}$ and $M_{h,j+1}$ in the ordering of the algorithm.  That is, writing $M_{i+1}=M_{h+1,j'}=N_{h+1,m'}$, we have $j' < \nu_{h+1}$.  Therefore, by the induction assumption, $M_{h+1,j'+1}=N_{h+1,m'+1}$, where $N_{h+1,m'+1}=X_{h+1}X_{w_1-1}X_{w_1}^{n_{w_1}} \ldots X_{w_s}^{n_{w_s}}$.  By the same arguments, for every $ p=1,\ldots, w_1-(h+1)$, $M_{h+1,j'+p}=N_{h+1, m'+p}$, where $N_{h+1,m'+p}=X_{h+1}X_{w_1-p}X_{w_1}^{n_{w_1}} \ldots X_{w_s}^{n_{w_s}}$.  With this explicit description of the elements $M_{h+1,j'+1},\ldots, M_{h+1,j'+w_1-(h+1)}$, we see that, by the instructions of the algorithm, $M_{h+1,j'+w_1-(h+1)}=X_{h+1}^2X_{w_1}^{n_{w_1}}\ldots X_{w_s}^{n_{w_s}}$ precedes $M_{h,j+1}$ in the ordering of $S$.  Write $M_{h+1,j'+w_1-(h+1)}=M_{i'}$.  Now according to the algorithm, $M_{i'+1}=X_hX_{w_1-1}X_t^{n_{w_1}-1}X_{w_2}^{n_{w_2}} \ldots X_{w_s}^{n_{w_s}}$, and this is perforce $M_{h,j+1}$.  As this element is also $N_{h,m+1}$, the claim is shown.

\emph{Case 2:}  $2 \leq n_h \leq h$.  We induct on $h$.  For $h=1$ there is nothing to show, since there is no element of $S_1$ with $2 \leq n_h \leq h$.  Let $h >1$, and assume that, for all $j'< \nu_{h-1}$ such that $X_{h-1}^2 \mid M_{h-1,j'}$ and $M_{h-1,j'} \neq X_{h-1}^h$, if $M_{h-1,j'}=N_{h-1,m'}$, then $M_{h-1,j'+1}=N_{h-1,m'+1}$.  If $n_h=2$, then $M_{i+1}= X_{h-1}X_{w_1-1}X_{w_1}^{n_{w_1}-1}X_{w_2}^{n_{w_2}}\ldots X_{w_s}^{n_{w_s}}$, and $M_{i+2}= X_hX_{w_1-1}^2X_{w_1}^{n_{w_1}-1}X_{w_2}^{n_{w_2}}\ldots X_{w_s}^{n_{w_s}}=N_{h, m+1}$
 and this is $M_{h,j+1}$ as claimed. 

If $n_h \geq 3$, then $M_{i+1}=X_{h-1}^2X_h^{n_h-3}X_{w_1}^{n_{w_1}}\ldots X_{w_s}^{n_{w_s}}$.  Write $M_{i+1}=M_{h-1,j'}=N_{h-1,m'}$.  In this case $M_{i+2} \in S_{h-2}$.  As above, this implies that $j' < \nu_{h-1}$ by Remark \ref{alg};  indeed, $j'+p < \nu_{h-1}$ while $X_{h-1}^2 \mid M_{h-1,j'+p}$.  Hence the induction assumption shows that $M_{h-1,j'+p}=N_{h-1,m'+p}=X_{h-1}^{p+2}X_h^{n_h-3-p}X_{w_1}^{n_{w_1}}\ldots X_{w_s}^{n_{w_s}}$
 for $p=1,\ldots,n_h-3$.  Note that, since we assume that $n_h \leq h$, $M_{h-1,j'+n_h-3}=X_{h-1}^{n_h-1}X_{w_1}^{n_{w_1}} \ldots X_{w_s}^{n_{w_s}} \neq X_{h-1}^h$.  Hence by the induction assumption we also have $M_{h-1,j'+n_h-2}=N_{h-1,m'+n_h-2}=X_{h-1}X_{w_1-1}^{n_h-1}X_{w_1}^{n_{w_1}-1}X_{w_2}^{n_{w_2}} \ldots X_{w_s}^{n_{w_s}}$.  As above, this shows that $M_{h-1,j'+n_h-2}$ precedes $M_{h,j+1}$ in the ordering of $S$.  Write \\ $M_{h-1,j'+n_h-2}=M_{i'}$.  Then $M_{i'+1}=X_hX_{w_1-1}^{n_h}X_{w_1}^{n_{w_1}-1}X_{w_2}^{n_{w_2}}\ldots X_{w_s}^{n_{w_s}}=N_{h,m+1}$, and this is $M_{h,j+1}$ as claimed. 

\emph{Case 3:} $n_h=h+1$.  This case we show to be impossible.  Suppose not, and let $h$ be the smallest number for which $M_{h,j}=X_h^{h+1}=M_i$ for some $j < \nu_h$.  Note that $h=1$ is impossible, since the algorithm explicitly states that $X_1^2=M_{1,\nu_1}$.  Similarly if $M_i=X_2^3$, then $M_{i+1}=X_1^2$ and the algorithm finishes, so that $h=2$ is impossible.  Thus we have $M_{i+1}=X_{h-1}^2X_h^{h-2}$.  Write $M_{i+1}=M_{h-1,j'}=N_{h-1,m'}$.  As in Case 2, one sees that $M_{h-1,j'+h-2}=X_{h-1}^h$ precedes $M_{h,j}$, and this contradicts the minimality of $h$.
\end{proof}

\begin{cor} Algorithm \ref{J-alg} finishes in a finite number of steps.

\end{cor}

\begin{proof} The fact that $S_h$ and $T_h$ are equal as ordered sets for all $h=1,\ldots,g$ shows that no element of $S$ is selected by the algorithm more than once.  Therefore, the number of steps in the algorithm is equal to the size of $S$, where $\lvert S \rvert= \lvert S_1 \cup \ldots \cup S_g \rvert = \lvert T_1 \cup \ldots \cup T_g \rvert$, and this is smaller than the number of monomials in $R$ of degree at most $g+1$.
\end{proof}

Lemma \ref{Sh-Th} also yields the following immediate observation:

\begin{cor}\label{in S}  With notation as above, $(X_1,\ldots,X_g) \mathfrak{m}^{g} \subset (S)$.

\end{cor}

\begin{proof} We have $S=T_1 \cup \ldots \cup T_g$.  Now observe that every monomial in $X_h \mathfrak{m}^g$ is a multiple of some element of $T_h$.
\end{proof}

The next result concerns elements of different degree, and completes the description of the ordering on $S$.

\begin{lem}\label{Sh-Shk} Let $S_h$ be as above, let $M_i \in S_h$, and write $M_i=X_hX_{v_1}\ldots X_{v_h}$, where $v_1 \leq \ldots \leq v_h$.  Then for each $k \geq 1$,  the element of $S_{h-k}$ immediately following $M_i$ in the ordering of $S$ is $ X_{h-k}X_{v_{k+1}-1} X_{v_{k+2}} \ldots X_{v_h}$.

\end{lem}

\begin{proof}  We induct on $k$.  For $k=1$, write $M_{i'}=X_h^2X_{v_2} \ldots X_{v_h}$, and observe that $M_i$ precedes $X_h^2X_{v_2} \ldots X_{v_h}$ in the ordering of $S$, or $M_i=M_{i'}$.  Furthermore, by the instructions of Algorithm \ref{J-alg}, there is no element of $S_{h-1}$ between $M_i$ and $M_{i'}$ in the ordering of $S$.  Therefore $M_{i'+1}= X_{h-1}X_{v_2-1}X_{v_3} \ldots X_{v_h}$ is the element of $S_{h-1}$ immediately following $M_i$.  Now let $k >1$ and assume the claim holds for $k-1$.  By the induction assumption, $M_{i'}= X_{h-k+1}X_{v_{k}-1} X_{v_{k+1}} \ldots X_{v_h}$ is the element of $S_{h-(k-1)}$ immediately following $M_i$ in the ordering of $S$.  Therefore, by Remark \ref{alg}, there is no element of $S_{h-k}$ which falls between $M_i$ and $M_{i'}$ in the ordering of $S$.  Now write $M_{i''}=X_{h-k+1}^2X_{v_{k+1}} \ldots X_{v_h}$.  As before, either $M_{i'}$ precedes $M_{i''}$ in the ordering of $S$ or $M_{i'}=M_{i''}$, and there is no element of $S_{h-k}$ between $M_{i'}$ and $M_{i''}$.  Therefore $M_{i''+1}=X_{h-k}X_{v_{k+1}-1}X_{v_{k+2}} \ldots X_{v_h}$ is the element of $S_{h-k}$ immediately following $M_{i'}$, and hence $M_i$.
\end{proof}

We may now proceed with our induction.  Proposition \ref{Im J} is an immediate consequence of the following result.

\begin{prop}\label{IS J} With notation as above, $I \cap (S)_{g+1} \ \subset \ J$.

\end{prop}

\begin{proof} 

We must show that every degree $g+1$ multiple of an element $M_i \in S$ which is in $I$ is also in $J$.  We induct on $i$.  Since $M_1=X_1X_d \in J$, obviously every multiple of $M_1$ is in $J$.  Now let $i >1$, and assume that, for every $i' < i$, $M_{i'} \omega' \in J$ for every monomial $\omega'$ with $M_{i'} \omega' \in (S)_{g+1} \cap I$.  We will show that $M_i \omega \in J$ for every monomial $\omega$ with $M_i \omega \in (S)_{g+1} \cap I$, and this will complete the proof.  

Let $M_i \in S_h$, and write $M_i=X_hX_{v_1}\ldots X_{v_h}$, where $h \leq v_1 \leq \ldots \leq v_h$.  If $h \leq g-1$, then $X_hX_{v_1} \in I$ since $X_h \mathfrak{m} \subset I$ for all $h=1,\ldots, g-1$.  On the other hand, if $h=g$, then $\deg(M_i)=g+1$, so our assumption is that $M_i$ itself is in $I$.  In this case, since $I$ is strongly stable of degree two, we still obtain $X_hX_{v_1} \in I$.  Thus $X_hX_{v_1}$ is a term of a unique element $f \in J$, given by $f=\sum_{k=-h+1}^\beta X_{h+k}X_{v_1+k}, \mbox{ where } \beta=\max\{b: X_{h+b}X_{v_1+b} \in I \} \geq 0$.  Fix $\omega$, and write $\omega= X_{w_1} \ldots X_{w_{g-h}}$, where $w_\alpha \leq w_{\alpha+1}$ for all $\alpha$.  We will show that $X_{h+k}X_{v_1+k}X_{v_2}\ldots X_{v_h} \omega \in J$ for every $k=-h+1,\ldots, \beta$ with $k \neq 0$, and hence that $M_i \omega =(X_{\nu_2} \ldots X_{\nu_h}\omega) f - \sum_{k \neq 0}X_{h+k}X_{v_1+k}X_{v_2}\ldots X_{v_h} \omega  \in J$.

Note that $X_{h+k}X_{v_1+k} \in I$ for each $k= -h+1,\ldots, \beta$.  Therefore, it suffices to show, for each $k \neq 0$, that $X_{h+k}X_{v_1+k}X_{v_2}\ldots X_{v_h} \omega$ is a multiple of an element of $S$ which precedes $M_i$ in the order of the algorithm.  Fix $k \neq 0$.  Write $v_2=h+u \geq h$.  If $h < g$, write $w_1=h+t$ and set $n= \min\{t,k,u\}$.  If $h=g$, then set $n=\min\{k,u\}$.  Suppose first that $n \leq -1$, and write $m=-n$.  By Lemma \ref{Sh-Shk}, the element of $S_{h-m}$ immediately following $M_i$ in the ordering of $S$ is $\Phi_m:=X_{h-m}X_{v_{m+1}-1} X_{v_{m+2}} \ldots X_{v_h}$.  Consider the element $\Upsilon_m:=X_{h-m}X_{v_{m+1}} X_{v_{m+2}} \ldots X_{v_h}$.  Note that $\Upsilon_m \in T_{h-m}$, and $\Upsilon_m$ is smaller than $\Phi_m$, in the reverse lexicographic ordering.  Hence by Lemma \ref{Sh-Th}, $\Upsilon_m$ is an element of $S_{h-m}$ which precedes $\Phi_m$, and hence $M_i$ in the algorithm.  Note that $\Upsilon_m \mid X_{h+k}X_{v_1+k}X_{v_2} \ldots X_{v_h}\omega$,
 since either $X_{h-m}=X_{h+k}$ or $X_{h-m}=X_{w_1}$, and since we have assumed that $m+1 \geq 2$.  Therefore $X_{h+k}X_{v_1+k}X_{v_2} \ldots X_{v_h}\omega \in J$ by the induction assumption.  Suppose next that $n=0$, so that $n=t$ or $n=u$. If $n=t$, so that $w_1=h$, then the element $\Upsilon_0:=X_hX_{v_1+k}X_{v_2}\ldots X_{v_h} \in S$ divides $X_{h+k}X_{v_1+k}X_{v_2}\ldots X_{v_h} \omega$.  Furthermore, $\Upsilon_0$ is smaller than $M_i$ in the reverse lexicographic ordering since $k \geq 1$.  If $n=u < t$, so that $w_1 > v_2$, then we may take $\Upsilon_0:=X_hX_{v_1+k}X_{w_1}X_{v_3}\ldots X_h$.  

Finally, suppose that $n \geq 1$.  Set $\Upsilon_n:=X_{h+k}X_{v_1+k}X_{v_2}\ldots X_{v_h} X_{w_1} \ldots X_{w_n}$.  We will show that $\Upsilon_n$ precedes $M_i$ in the order of the algorithm, and hence that $\Upsilon_n w_{n+1} \ldots w_{g-h} \in J$ by the induction assumption.  Since $w_1, v_2 \geq h+n$, we have $\Upsilon_n \in S_{h+n}$.  Write $\Upsilon_n= X_{h+n}X_{b_1} \ldots X_{b_{h+n}} \mbox{, where } h+n \leq b_1 \leq \ldots \leq b_{h+n}$.
By Lemma \ref{Sh-Shk}, the element of $S_h$ immediately following $\Upsilon_n$ in the ordering of the algorithm is $\Phi:=X_hX_{b_{n+1}-1}X_{b_{n+2}} \ldots X_{b_{h+n}}$.  We claim that either $M_i=\Phi$, or $\Phi$ precedes $M_i$ in the order produced by the algorithm.  If $v_1+k>v_2$ or $ w_n>v_2$, then, for some $j=2,\ldots, h$, we will have $b_{n+j}> v_j$ and $b_{n+j'} =v_{j'}$ for all $j' \geq j$, making $\Phi$ greater than $M_i$ in the reverse lexicograhic order.  So assume this is not the case.  If $w_n > v_1+k$, then $b_{n+1}-1=w_n-1 > v_1+k-1 >v_1$, again making $\Phi$ greater than $M_i$.  Otherwise, $b_{n+1}-1=v_1+k-1 \geq v_1$, and $\Phi$ is greater than or equal to $M_i$.  This completes the proof.
\end{proof}

The proof of Proposition \ref{Im J} now follows immediately.

\begin{proof} (of Proposition \ref{Im J})  Observe that $I\mathfrak{m}^{g-1}=I \ \cap \ (X_1,\ldots,X_g)\mathfrak{m}^{g}$ by degree reasons.  The result now follows from Corollary \ref{in S} and Proposition \ref{IS J}
\end{proof}

Proposition \ref{IS J} also shows that $J$ is a reduction of $I$, and leads to a bound for the reduction number of $I$ with respect to $J$.  See also \cite{CNPY}.

\begin{prop}\label{J red}  If $J$ is the diagonal ideal associated to $I$ as above, then $I^g = JI^{g-1}$.  That is, $J$ is a reduction of $I$ with reduction number at most $g-1$.

\end{prop}

\begin{proof}  As before we write $I=(X_1,\ldots,X_{g-1})\mathfrak{m} + X_g(X_g,\ldots, X_\nu)$, where $g \leq \nu \leq d$.  Note that, by degree reasons, Proposition \ref{Im J} actually shows that $I \mathfrak{m}^{g-1} \subset J \mathfrak{m}^{g-1}$.  Therefore, since $(X_1,\ldots,X_{g-1})\mathfrak{m} \subset I$,
\begin{align}
& I \big[ (X_1,\ldots,X_{g-1}) \mathfrak{m} \big]^{g-1} \subseteq J \big[ (X_1,\ldots,X_{g-1}) \mathfrak{m} \big]^{g-1} \subset JI^{g-1}.
\end{align}
Consider an element $N \in I^g$.  Write $N=X_{i_1}X_{j_1}X_{i_2}X_{j_2}\ldots X_{i_g}X_{j_g}$, where $X_{i_t}X_{j_t} \in I$ for all $t=1,\ldots,g$ and $i_1,\ldots,i_g \in \{1,\ldots, g\}$.  Additionally, we require that the $i_k$'s and $j_k$'s are chosen in such a way that $b(N):= \left| \{ t \mid i_t=g \} \right|$ is as small as possible. 
We induct on $b=b(N)$ to show that $N \in JI^g$.  If $b \leq 1$, then $N \in I [(X_1,\ldots,X_{g-1})\mathfrak{m}]^{g-1}$, hence equation (2) shows that $N \in JI^{g-1}$.  Now let $b > 1$, and assume that for any monomial $M \in I^g$ with $b(M) < b(N)$, $M \in JI^{g-1}$.  Choose $t$ with $i_t=g$, and write $j_t=s$.  Consider the element $f:=\sum_{k=0}^{g-1}X_{g-k}X_{s-k} \in J$.  For each $k$, set $M_k:=(N/X_gX_s)X_{g-k}X_{s-k} \in I^g$.
Note that for each $k >0$, $b(M_k) < b(N)$, hence $M_k \in JI^{g-1}$ by the induction assumption, while $M_0=N$, which shows that $N \in J$.
\end{proof}

In light of Proposition \ref{J red}, from now on we shall refer to $J$ as the \emph{diagonal reduction} of $I$.  

\end{section}

\begin{section}{A Formula for the Core}

This section will be devoted to proving our main result, which is as follows:

\begin{thm}\label{core thm} Let $I \subset R$ be a strongly stable ideal of degree two which has the $G_d$ property.  Let $\h I=g$, and let $\mathfrak{m}=(X_1,\ldots,X_d)$.  Then $\core=I \mathfrak{m}^{g-1}$.

\end{thm}

We shall prove the inclusion $\core \subseteq I \mathfrak{m}^{g-1}$ by reducing to the $\mathfrak{m}$-primary case, then computing the socle of the diagonal reduction $J$ by viewing it as a Northcott matrix.  We perform this computation, then give the proof of Theorem \ref{core thm}.  Recall the following fact about Northcott ideals, which were first studied in \cite{DG}.  Let $K=(a_1,\ldots,a_n)$ and $L=(b_1,\ldots,b_n)$ be $R$-ideals generated by regular sequences, with $L \subset K$.  Let $A$ be an $n \times n$ matrix with 
\[ A
\begin{pmatrix}
a_1 \\
\vdots \\
 a_n
\end{pmatrix}
=
\begin{pmatrix}
b_1 \\
\vdots \\
 b_n
\end{pmatrix}.
\]
Then $L:K=( \det(A), L)$.  We shall use this fact to compute Soc$(R/J)$, where $J$ is the diagonal reduction of the strongly stable ideal $\mathfrak{m}^2$.  We first give two easy lemmas which we shall use in our computation of the socle.   

\begin{lem}\label{in soc}  Let $J \subset R$ be the diagonal reduction of $\mathfrak{m}^2$, where $\mathfrak{m}=(X_1,\ldots,X_d)$.  Then for each $h=1,\ldots,d$, $X_1^{h}X_{d-h+1} \in J$.  

\end{lem}

\begin{proof}  We induct on $h$, the case of $h=1$ being clear.  Let $h > 1$, and assume the claim holds for all $h'<h$.  The element $\phi:= X_1^{h-1}(X_1X_{d-h+1}+ X_2X_{d-h+2}+ \ldots +X_hX_d)$ is in $J$, and by the induction assumption each term of $\phi - X_1^hX_{d-h+1}$ is in $J$.  Therefore $X_1^hX_{d-h+1} \in J$ 
\end{proof}

The next result will allow us to induct on the dimension $d$ of the ring $R$.

\begin{lem}\label{soc lem}  Let $J \subset R$ be as above.  Let $S=k[X_1,\ldots,X_{d-1}]$ and let $K \subset S$ be the diagonal reduction of $(X_1,\ldots,X_{d-1})^2$.  Then for every $\alpha \in K$, $X_1\alpha \in J$.  

\end{lem}
 
\begin{proof} Write $J=(f_1,\ldots,f_{d})$, where $f_i=\sum_{j=1}^{d-i+1}X_jX_{j+i-1}$, and write $K=(g_1,\ldots,g_{d-1})$, where $g_i=\sum_{j=1}^{d-i}X_jX_{j+i-1}$.  Observe that for each $i=1,\ldots, d-1$, $f_i=g_i+X_{d-(i-1)}X_d$.  Thus for each $i$, $X_1g_i=X_1f_i-X_1(X_{d-(i-1)}X_d)=X_1f_i-(X_1X_d)X_{d-(i-1)} \in J$.  
\end{proof}

We are now ready to compute Soc$(R/J)= (J:\mathfrak{m})/J$.  

\begin{prop}\label{socle J} Let $J \subset R$ be the diagonal reduction of $\mathfrak{m}^2$, where $\mathfrak{m}=(X_1,\ldots,X_d)$.  Then Soc$(R/J)=( X_1^d+J)/J$.

\end{prop}

\begin{proof}  This is clear in the case of $d=1$, so let $d \geq 2$.  Write $J=(f_1,\ldots,f_d)$.  Since $J:\mathfrak{m}$ is a Northcott ideal, it suffices to compute $\det(A)$ up to equivalence$\pmod{J}$, where $A$ is the $(d \times d)$ matrix such that 
\[ A
\begin{pmatrix}
X_1 \\
\vdots \\
 X_d
\end{pmatrix}
=
\begin{pmatrix}
f_1 \\
\vdots \\
 f_d
\end{pmatrix}.
\]
For conciseness, we shall refer to such a matrix as \emph{the matrix corresponding to $J$}.  Let $(a_{ij})$ be the $d \times d$ matrix given by $a_{ij}=\frac{1}{2}\left(X_{j-(i-1)}+X_{j+(i-1)}\right), \ \ i,j=1,\ldots, d$,
where we adopt the convention that $X_{j-(i-1)}=0$ if $j-(i-1) < 1$ and $X_{j+(i-1)}=0$ if $j+(i-1) > d$.  We claim that $A=(a_{ij})$, that is, for each $i=1,\ldots, d$, $\sum_{j=1}^d a_{ij}X_j=f_i$.  Recall that, for each $i=1,\ldots,d$, $f_i=\sum_{j=1}^{d-(i-1)}X_jX_{j+(i-1)}$.  Therefore, for each $i$, 
\begin{align}
\sum_{j=1}^d a_{ij}X_j= & \sum_{j=1}^d\left(\frac{1}{2}X_{j-(i-1)}\right)X_j+ \sum_{j=1}^{d}\left(\frac{1}{2}X_{j+(i-1)}\right)X_j \notag \\
= & \sum_{j=i}^d\left(\frac{1}{2}X_{j-(i-1)}\right)X_j+ \sum_{j=1}^{d-(i-1)}X_j\left(\frac{1}{2}X_{j+(i-1)}\right) \notag \\
 = & \frac{1}{2} \left(\sum_{j=1}^{d-(i-1)}X_{j}X_{j+(i-1)} \right)+ \frac{1}{2}\left( \sum_{j=1}^{d-(i-1)}X_jX_{j+(i-1)}\right) = f_i, \notag 
\end{align}  
which shows the claim.  We will prove the following:
\begin{align}
\det(A) & \equiv X_1^d \pmod{J}, \mbox{ and } \\
X_1^d & \equiv (-1)^cX_d^d \! \pmod{J}, \mbox{ where } c=
\begin{cases}
0, \ d \equiv 0 \mbox{ or } 1 \! \pmod{4} \\
1, \ d \equiv 2 \mbox{ or } 3 \! \pmod{4}
\end{cases}.
\end{align}

We induct on $d$.  If $d=2$, we have $A= \left( \begin{array}{c c} X_1 & X_2 \\ \frac{1}{2}X_2 & \frac{1}{2}X_1 \end{array} \right)$ and $J=(X_1^2+X_2^2,X_1X_2)$.  Thus $\det A= \frac{1}{2}(X_1^2-X_2^2) \equiv X_1^2 \equiv -X_2^2 \pmod{J}$ as claimed.

Now let $d>2$, and assume the equivalences in equations (3) and (4) hold in a polynomial ring in $d-1$ variables.  Let $K$ be the diagonal reduction of the ideal $(X_1,\ldots,X_{d-1})^2$ in the ring $S=k[X_1,\ldots,X_{d-1}]$, and let $B$ be the matrix corresponding to $K$.  Let $L$ the diagonal reduction of $(X_2,\ldots,X_d)^2$ in the ring $T=k[X_2,\ldots,X_d]$, and let $C$ be the matrix corresponding to $L$.  Let $A'$ be the upper right hand $(d-1) \times (d-1)$ minor of $A$, and let $A''$ be the upper left hand $(d-1) \times (d-1)$ minor of $A$. Observe that each entry of the $(d-1) \times (d-1)$ matrix $A''-B$ is a multiple of $X_d$.  Therefore $\det(A'')-\det(B)$ is a multiple of $X_d$.  Similarly $\det(A')-\det(C)$ is a multiple of $X_1$.  Therefore $X_1\det(A'') \equiv X_1\det(B) \pmod{J}$, and $X_d \det(A') \equiv X_d\det(C) \pmod{J}$, since $X_1X_d \in J$.  The $d$th row of $A$ is $( \frac{1}{2}X_d ,0, \ldots,  0, \frac{1}{2}X_1 )$.  Therefore $\det(A)= \frac{1}{2}(X_d \det(A') + (-1)^{d-1} X_1 \det(A''))
 \equiv \frac{1}{2}( X_d\det(C) + (-1)^{d-1} X_1\det(B)) \pmod{J}$.  By the induction assumption, $\det(C)\equiv  (-1)^{c'}X_d^{d-1} \pmod{K}$ and $\det(B) \equiv X_1^{d-1} \pmod{L}$, where
\[c'=
\begin{cases}
0, \ d-1 \equiv 0 \mbox{ or } 1 \! \pmod{4} \\
1, \ d-1 \equiv 2 \mbox{ or } 3 \! \pmod{4}
\end{cases}. \]   
By Lemma \ref{soc lem}, $X_1 \det(B) \equiv X_1^{d} \pmod{J}$ and $X_d\det(C) \equiv  (-1)^{c'}X_d^{d} \pmod{J}$ (through symmetry);  hence 
\begin{align}
\det(A) \equiv  \frac{1}{2}\left( (-1)^{c'}X_d^{d}+(-1)^{d-1}X_1^d \right) \pmod{J}.
\end{align}  

Now fix $i$ and consider the element $\Psi_i:= X_1^{d-i-1}X_2^{i-1}(X_1X_{i} +X_2X_{i+1} +X_3X_{i+2}+ \ldots+X_{d+1-i}X_d) \in J$.
Recall that by Proposition \ref{in soc}, for each $h=1,\ldots,d$, $X_1^{h}X_{d-h+1} \in J$.  In particular, $X_1^{d-i-1}X_t\in J$ for all $t \geq i+2$, so that all but the first two terms of $\Psi_i$ are in $J$.  Thus, for each $i=1,\ldots,d-1$, $X_1^{d-i}X_2^{i-1}X_i \equiv -X_1^{d-i-1}X_2^iX_{i+1} \pmod{J}$.  This gives a string of $d-1$ congruences $X_1^d \equiv -X_1^{d-2}X_2^2 \equiv X_1^{d-3}X_2^2X_3 \equiv X_1^{d-4}X_2^3X_4 \equiv \ldots \equiv (-1)^{d-1}X_2^{d-1}X_d \pmod{J}$.
Furthermore, by the induction assumption, $X_2^{d-1} \equiv (-1)^{c'}X_d^{d-1} \pmod{L}$.  Therefore, by Lemma \ref{soc lem}, $X_1^d \equiv (-1)^{d-1}X_2^{d-1}X_d \equiv (-1)^{d-1+c'}X_d \pmod{J}$,
and it is straightforward to verify that $(-1)^{d-1+c'}=(-1)^c$ for any value of $d \pmod{4}$.  We have now proven the equivalence from equation (4).  Finally, the equivalence of (3) follows from equations (4) and (5).
\end{proof}

We are now ready to prove Theorem \ref{core thm}.

\begin{proof} (Of Theorem \ref{core thm})  We first show that $I \mathfrak{m}^{g-1} \subseteq $ mono$(K)$, where $K$ is a general minimal reduction of $I$.  By Theorem \ref{ANI} and \cite{Monom}, this will imply that $I \mathfrak{m}^{g-1} \subseteq \core$.  (The statement of the theorem of Polini and Ulrich which we cite here is given in Section 2.)  Since $I \mathfrak{m}^{g-1}$ is a monomial ideal, it suffices to show that $I \mathfrak{m}^{g-1} \subseteq K$, or equivalently, that $\mathfrak{m}^{g-1} \subseteq K:I$.  We will show that the Hilbert function of $R/(K:I)$ is zero in degree $g-1$.  Let $J$ be the diagonal reduction of $I$ which we studied in Section 3.  As $J$ and $K$ are both minimal reductions of $I$,  by \cite{Hadley} (see also the rephrasing of this theorem in \cite{MarkJ}), the ideals $K:I$ and $J:I$ are both $d$-residual intersections of $I$.  Note also that $J$ and $K$ are both generated in degree 2.  Therefore,by \cite{CEU}, $R/(K:I)$ and $R/(J:I)$ have the same Hilbert series.  Therefore it suffices to show that $\mathfrak{m}^{g-1} \subset J:I$, which is precisely what we have shown in Proposition \ref{Im J}.  

We now prove that $\core \subseteq I\mathfrak{m}^{g-1}$.  Note that $I\mathfrak{m}^{g-1}=I \cap \mathfrak{m}^{g+1}$ by degree reasons, and recall that core$(I) \subseteq I$ for any ideal $I$.  Therefore, it suffices to show that core$(I) \subseteq \mathfrak{m}^{g+1}$.  Equivalently, we must show that core$(I)$ contains no element of degree $g$.  As the core of a strongly stable ideal is strongly stable by Proposition \ref{core ss}, it suffices to show that $X_1^g \notin $ $\core$.  We will show that $X_1^g$ is not contained in the diagonal reduction $J$ of $I$.  This reduces to showing that $X_1^g \notin J'$, the diagonal reduction of the ideal $(X_1,\ldots,X_g)^2 \subset R'=k[X_1,\ldots,X_g] \cong R/(X_{g+1},\ldots,X_d)$, as $J'=J+(X_{g+1},\ldots,X_d)/(X_{g+1},\ldots,X_d)$.  The ring $R'/J'$ is Artinian, and hence must have a nonzero socle.  Therefore, as Proposition \ref{socle J} shows that the socle of $R'/J'$ is $(X_1^g +J')/J'$, we conclude that $X_1^g \notin J'$.
\end{proof}

\end{section}

\section*{Acknowledgments}

This work was carried out largely while the author was a doctoral student at the University of Notre Dame.  She wishes to extend her deep gratitude to her thesis adviser, Claudia Polini, for all of her guidance and support, both mathematical and otherwise.  She is grateful to Bernd Ulrich for helpful mathematical discussions, and she thanks Claudia Polini and Bernd Ulrich for sharing a result of theirs in \cite{Monom} prior to its submission for publication.  Finally, she thanks her postdoctoral mentor, Alberto Corso, for suggesting this problem to her, and for much support and help in the final preparation of the manuscript since her arrival at the University of Kentucky.

\bibliographystyle{abbrv}

\bibliography{SS2_refs}

\end{document}